\documentclass[11pt]{article}
\usepackage{amssymb,amsmath,amsthm,a4wide}
\usepackage{url,graphics}

\newtheorem{theorem}{Theorem}

\newtheorem{lemma}[theorem]{Lemma}

\begin{document}

\title{Uniform van Lambalgen's theorem fails for computable randomness}
\author{Bruno Bauwens\footnote{
\newline \textbf{Address}: National Research University Higher School of Economics,
Faculty of Computer Science,
11 Pokrovsky Boulevard, Kab S833, Moscow 109028, Russia.
%\newline \textbf{Email}: brbauwens -at- gmail.com
  } 
}

\date{}

\vspace{-3em}
\maketitle  
\vspace{-1cm}

\begin{abstract}
  \noindent
  We show that there exists a bitsequence that is not computably random for which the odd bits are computably random 
  and the even bits are computably random relative to the odd bits. 
\end{abstract}

%Downey: wrong theorem: 7.1.13
\noindent
Imagine two referees need to decide whether a bitsequence is random.
The first referee inspects the odd bits. The second referee
inspects the even bits conditional to the odd ones.  Both referees claim that
their bits are random. Is it still possible that the full sequence is
non-random? We investigate this question for computable randomness. 
%one of the most natural definitions of randomness. 

\smallskip

A  {\em martingale} $f$ is a function mapping strings to non-negative reals such that 
$f(x) = (f(x0) + f(x1))/2$ for all $x$. Let $\alpha$ and $\beta$ be sequences.
$\alpha$ is {\em computably random} or  {\em CR} (respectively, {\em CR relative
to}~$\beta$) if for every totally computable martingale~$f$ (respectively,
totally computable relative to~$\beta$), the set of values of $f$ on all initial segments of
$\alpha$ is bounded.

Observe that the odd bits of a CR sequence define a sequence
that is also CR.  Suppose that the odd bits are
CR and the even bits are CR relative to the odd ones, 
is the sequence CR?
This question has been asked repeatedly~\cite{vanLambalgenDemuth,RuteSchnorr,logicBlog2014},
and we give a negative answer.
%and the following was conjectured in~\cite{RuteSchnorr}.
%
\begin{theorem}\label{th:main}
  There exists a non-CR sequence  for which the odd bits are CR
   and the even bits are CR relative to the odd bits.
\end{theorem}
This is remarkable for three reasons. 
First, the result is positive for the closely related
notion of Schnorr randomness\footnote{
  A sequence $\alpha$ is 
  {\em Schnorr random} if there exists a computable non-decreasing unbounded function $h$ and a
  computable martingale $d$ such that $d(\alpha_1\dots \alpha_n) \ge h(n)$ for 
  infinitely many~$n$}~\cite{StephanFranklinVanLambalgen,RuteSchnorr}. 
(See the footnotes for details.)
Secondly, with a slightly stronger assumption the answer is positive: 
%$\alpha = \alpha_1\alpha_2\dots$ and $\beta = \beta_1\beta_2\dots$:
If $\alpha$ is CR  relative to $\beta$ and $\beta$ is CR relative to $\alpha$, 
then $\alpha_1\beta_1\alpha_2\beta_2\dots$ is CR.\footnote{
  Indeed, any martingale is the product of 
  a martingale that only bets on odd bits and one that bets on even bits. 
  This decomposition can happen in a computable way. The proof finishes by
  a simple transformation of these martingales to conditional martingales.
  The converse of the statement is also true for uniformly conditional randomness, 
  see further.
  }
Thirdly, it has been repeatedly claimed that a positive answer 
follows by the same argument as for Martin-L\"of randomness. 
In~\mbox{\cite[Remark 3.2]{RuteSchnorr}} it is explained why this is not true, 
and  Theorem~\ref{th:main} was conjectured.\footnote{
  Remark 3.2 considers Schnorr randomness, but is also valid for computable randomness.
  The conjecture is located in the paragraph below Theorem 1.3.
  }
  %\,\footnote{
  %Also a weaker form of theorem has been conjectured, in which the even bits 
  %are only required to be  {\em uniformly} CR relative to the odd bits. See the next footnote
  %for the definition.
  %}

\medskip
Van Lambalgen's requirement for randomness is that $\alpha_1\beta_1\alpha_2\beta_2\dots$ is random
if and only if, $\alpha$ is random and $\beta$ is random relative to $\alpha$. 
Both for computable and Schnorr randomness the forward implication is known to fail~\cite{vanLambalgenFails}.
%there exists a random $\alpha_1\beta_1\alpha_2\beta_2\dots$ 
%for which $\alpha$ is not random relative to $\beta$~\cite{vanLambalgenFails}.
However, for both notions of randomness, this forward direction holds when the uniform\footnote{
  A function $f$ is {\em uniformly} computable relative to~$\alpha$ if there exists 
  an oracle Turing machine $U$ such that $U^\alpha = f$ and $U^\beta(x)$ is defined for all~$\beta$ and~$x$. 
  A sequence is {\em uniformly random relative to $\alpha$} if no martingale that is uniformly computable 
  relative to~$\alpha$ has unbounded values on it. We refer to~\cite{timeBoundedVanLambalgen} 
  for more variants.
  }
variant of relative randomness is used~\cite{RuteSchnorr}.
From Theorem~\ref{th:main}, it follows follows that the reverse direction does not hold 
for computable randomness. In fact, this already follows from the weaker uniform variant, 
in which the even bits are only required to be uniformly random relative to the odd ones.
In conclusion: The uniform van Lambalgen's criterion holds for Schnorr randomness 
and fails for computable randomness. 
This suggests that Schnorr randomness is more fundamental, 
and this was also observed in applications 
in computable analysis and reverse mathematics~\cite{ruteTalkCCR2015}.

%Hence by the result from~\cite{StephanFranklinVanLambalgen,RuteSchnorr} mentioned above, the uniform 
%van Lambalgen's requirement holds for Schnorr randomness.
%Many researchers have wondered whether the uniform version 
%of the requirement also holds for computable randomness. 
%Our result implies a negative answer.

\section*{Proof}

We use a technique that was historically introduced to prove 
the G\'acs--Ku\v{c}era theorem~\mbox{\cite{everySequenceFromRandomOne,Kucera1985measure}}. 
It allows to construct a sequence that encodes another sequence, and for which
some martingales remain bounded, see~\cite[Lemma 8.3.1 p325]{Downey}.
This technique uses the following lemma. 

\begin{lemma}\label{lem:coding}
  For each martingale $d$, each string $x$ and each natural number $s$, there exist at least two strings $y$ 
  of length $s+2$ such that $d(xy)/d(x) < 1 + 2^{-s}$.
\end{lemma}
\begin{proof}
 For all $0 < \varepsilon \le 1/2$, at most a fraction $1/(1+\varepsilon)$ of strings $y$ of a given length 
 satisfy $d(xy)/d(x) \ge 1+\varepsilon$. Hence, more than a fraction $1-1/(1+\varepsilon) > \varepsilon/2$ does
 not satisfy this property. The amount of such strings of length $s+2$ is at least $ 2^{s+2} \varepsilon/2$ which equals $2$ for $\varepsilon = 2^{-s}$.
\end{proof}

\begin{proof}[Proof of  Theorem~\ref{th:main}.]
  The definition of CR-randomness does not change if we use only rational martingales~\cite[Prop. 7.1.2 p270]{Downey}.
 We choose $\alpha$ to be any computably random sequence. 
 Let $d_1, d_2, \dots$ be an enumeration of all 
 rational partial martingales that are partial computable with oracle~$\alpha$.
 Let~$\varepsilon_s = 2^{-s}$.
 $\beta$ is constructed (in a non-computable way) in stages, together with a total martingale~$d$.

 \medskip
 \noindent
 Initially, $\beta$ is the empty string and $d = 1$. 
 At each stage $s \ge 1$ we update $\beta$ and $d$:
 \begin{itemize}
   \item 
     If $d_s$ is total relative to the oracle $\alpha$ and $d(\beta)$ is positive, we replace $d$ by $d + \frac{\varepsilon_s}{d_s(\beta)}d_s$.
     Otherwise, $d$ is unchanged.
   \item 
     To $\beta$ we append either the lexicographically first or second string $y$ of length 
     $s+2$ such that 
     \begin{equation}\label{eq:martingaleIncrease}\tag{*}
       \frac{d(\beta y)}{d(\beta)} < 1 + \varepsilon_s 
     \end{equation}
     (by  Lemma~\ref{lem:coding} 
     there exist at least two such strings). 
     The choice depends on whether $d_{s+1}$ is total relative to oracle~$\alpha$.
   \item 
     To $\beta$ we again append either the lexicographically first or second string $y$ of length 
     $s+2$ such that the inequality above holds.
     The choice now depends on the value of $\alpha_t$ for some fresh value $t$. More precisely, let $t$ be the minimal value 
     that exceeds $|\beta y|$, the previous value of $t$ (for stages $s\ge 2$), 
     and all computation times and uses of the oracle $\alpha$ in the evaluations of $d(z)$
     for all strings $z$ of length at most $|\beta y|$. 
     %If $t \le |\beta y|$ or does not exceed its value
     %from the previous step, then we increase $t$ minimally such that these conditions are satisfied.
 \end{itemize}
 %For this construction we need to assume that in each stage the value of $t$ is
 %strictly above the value of $t$ from the previous stage.  This is natural,
 %because after each stage, $d$ becomes more complex and needs to be evaluated
 %on a larger set of strings. 
 %% (containing the $|\beta y|$-bit strings). 
 %For later use, we also assume that $t$ exceeds all values of $|\beta y|$ in the third step. Again this is natural, 
 %because $t$ is at least the evaluation time of $d$ on strings of length at most $|\beta y|$.
 End of construction.

 \medskip
 We show that $\beta$ is computably random relative to $\alpha$. 
 The value $d(\beta)$ after each stage $s$ is at most 
 $(1 + \varepsilon_1 + \dots + \varepsilon_s)\exp(2\varepsilon_1 + \dots + 2\varepsilon_s)$,
 and this has the finite limit $2\exp 2$. 
 Indeed, in each stage, after a possible update of $d$, the value $d(\beta)$ increases by $\varepsilon_s$, 
 and after each extension of $\beta$, the value increases by at most a factor $1+\varepsilon_s \le \exp \varepsilon_s$.

 Each test that is computable relative to~$\alpha$, appears in the sequence of
 tests. For each such test we have $d_s \le O(d)$, because at stage $s$,
 this test will be used to increment~$d$.  Hence, $d_s$ is bounded on infinitely many initial segments of $\beta$. 
 By the savings technique\footnote{
  If a martingale $f$ is unbounded on some sequence $\gamma$, then
  there exists a martingale that tends to infinity. For example: 
  $\tilde{f}(x) = \sum_k 2^{-k} f(x_k)$,
  %defines a martingale that tends to infinity on $\gamma$;
  %Here are the details. Let $f$ be a martingale. 
  where $x_k$ is equal to $x$ if no prefix $y$ of $x$ satisfies $f(y) \ge 2^{2k}$, and equal to the shortest such $y$ otherwise.
%  If $f$ is computable, then $\tilde{f}$ is computable.
%
%  converted to a computable martingale $\tilde f$ that tends to infinity on $\gamma$.
%  Indeed let $\tilde{f} = \sum_k 2^{-k} f_k$,
%  where for all $k$, the martingale $f_k$ is defined by induction for $b \in \{0,1\}$ by
%  \[
%   f_k(xb) = 
%   \begin{cases}
%     f_k(x)  &\text{ if } f_k(x) \ge 2^{2k} \\
%     f(xb) &\text{ otherwise.}
%   \end{cases}
%  \]
 } 
 this implies that no computable martingale is unbounded on $\beta$, i.e., 
 $\beta$ is computably random relative to~$\alpha$.

 \medskip
 It remains to show that the pair $\alpha_1\beta_1\alpha_2\beta_2\dots$ is not computably random.
 We construct a martingale $e$ that plays only on the bits 
 $\alpha_t$ that were used in the 3rd step of each stage in the construction, and on these bits the capital is doubled.
 In the evaluation of $e(a_1b_1\dots a_{n-1}b_{n-1}a_n)$, we pretend that $a$ and $b$ are initial segments of $\alpha$ and $\beta$, 
 and try to rerun the construction above.
 In this way, we hope to find the positions~$t$ corresponding to the bits~$\alpha_t$.

 For this, we need to know the function $d$ used in each stage. 
 %and need to evaluate it on all strings of some lengths.
 And for this, we need to know which functions among $d_1, d_2, \dots$ are total. 
 The key observation is, that knowing a prefix of $\beta$ we can decide the totality of the functions $d_1, d_2, \dots$
 Indeed, knowing the updated function $d$ in a stage $s$, 
 we can compute the lexicographically first and second string for which inequality~\eqref{eq:martingaleIncrease} holds,
 then observe which one is equal to the corresponding segment of $\beta$, and from this we know the 
 totality of $d_{s+1}$. This allows us to update $d$ in the next stage, and we can repeat this procedure. 
 By choice of~$t$, we are able to compute each index $t$ using the prefix $\alpha_1\beta_1\dots \alpha_{t-1}\beta_{t-1}$, 
 and from the 3rd stage of the construction, we obtain~$\alpha_t$.

 \medskip
 {\em Detailed construction of $e$.} Let $e(\text{empty string}) = 1$. 
 For $n$-bit $a$ and $b$, the values of $e(a_1b_1\dots a_n)$ and $e(a_1b_1\dots a_nb_n)$ 
 are equal and are defined by induction on $n$.
 %To evaluate $e(a,b)$, we assume that $b$ is long enough to go through the procedure below.\footnote{
 % It suffices that $|b| \ge 2\sum_{i=0}^{|a|} (i+2)$.
 % Indeed, the number of stages that need to be evaluated is at most $|a|$ because the sequence of values of $t$ 
 % is strictly increasing. 
 %}
 %If this is not the case, we compute the value by averaging over all long enough extensions of $b$.
 Let $a$ be of length $n$ and $b$ of length $n-1$.
 The value of $e(a_1b_1\dots a_{n-1}b_{n-1}a_n)$ is defined in stages.
 We start with $s=1$ and $d = 1$.

 At stage $s$ we evaluate the current function $d$ on all strings 
 of length at most $2\sum_{i=1}^s (i+2)$. Let $t$ be the minimal value that exceeds this length, the value of $t$ in the previous stage (if~$s \ge 2$) and 
 all the computation times and oracle uses in these evaluations of $d$.
 Note that~$t$ can be infinite if one of the guesses for the totality of $d_1, \dots, d_s$ was wrong. 
 We only need to discriminate between the following cases:
 \begin{itemize}
   \item Case $t < |a|$. 
     We check whether the appropriate segment of length $s+2$ of $b$ is indeed the lexicographically 
     first or second string $y$ satisfying~\eqref{eq:martingaleIncrease}. ($b$ is long enough because $|\beta y| < t \le |a| = |b|+1$.)
     If this is not true, we set $e(a_1b_1\dots b_{n-1}a_n) = e(a_1b_1\dots b_{n-1})$, (the value of $e$ does not matter here);
     otherwise, we continue the simulation and proceed to stage $s+1$ 
     where the function $d$ is updated according to the suggested totality of~$d_{s+1}$.
   \item Case $t > |a|$. We set $e(a_1b_1\dots b_{n-1}a_n) = e(a_1b_1\dots b_{n-1})$. 
   \item Otherwise we have $t = |a|$. 
     Let $i$ be the bit that is encoded in the 3rd step (which equals~$\alpha_t$ for correct inputs). 
     Let
     \[
       e(a_1b_1\dots b_{n-1}a_n) = 
       \begin{cases}
	 2e(a_1b_1\dots b_{n-1}) & \text{if $a_n = i$}
	 \\0 & \text{otherwise.}
       \end{cases}
     \]
 \end{itemize}
 $e$ is computable, and on initial segments of $\alpha_1\beta_1\alpha_2\beta_2\dots$, the martingale $e$ is unbounded. 
 $(\alpha,\beta)$ is not CR and the theorem is proven.
\end{proof}

\subsubsection*{Acknowledgments}
I am grateful to Jason Rute for bringing this question to my attention and for useful discussion.
I am grateful to the Heidelberg university for financial support through the
``Focus Semester on Algorithmic Randomness'' in June 2015.  I thank the organizers,
Wolfgang Merkle, Klaus Ambos-Spies, Nadine Losert, Martin Monath and the participants for creating
a nice work atmosphere.

\bibliography{kolmogorov}

\begin{thebibliography}{10}

\bibitem{timeBoundedVanLambalgen}
Diptarka Chakraborty, Satyadev Nandakumar, and Himanshu Shukla.
\newblock On resource-bounded versions of the van lambalgen theorem.
\newblock In {\em International Conference on Theory and Applications of Models
  of Computation}, pages 129--143. Springer, 2017.

\bibitem{vanLambalgenDemuth}
David Diamondstone, Noam Greenberg, and Dan Turetsky.
\newblock A van {L}ambalgen theorem for {D}emuth randomness.
\newblock In {\em Proceedings of the 12th Asian Logic Conference}, pages
  115--124, 2013.

\bibitem{Downey}
Rodney.~G. Downey and Dennis.~R. Hirschfeldt.
\newblock {\em Algorithmic Randomness and Complexity}.
\newblock Theory and Applications of Computability. Springer, 2010.

\bibitem{StephanFranklinVanLambalgen}
Johanna Franklin and Frank Stephan.
\newblock Van {L}ambalgen's theorem and high degrees.
\newblock {\em Notre Dame Journal of Formal Logic}, 52(2):173--185, 2011.

\bibitem{everySequenceFromRandomOne}
P{\'e}ter G{\'a}cs.
\newblock Every sequence is reducible to a random one.
\newblock {\em Information and Control}, 70(2/3):186--192, 1986.

\bibitem{Kucera1985measure}
Anton{\'\i}n Ku{\v{c}}era.
\newblock Measure, {$\Pi^0_1$}-classes and complete extensions of {PA}.
\newblock In {\em Recursion theory week}, pages 245--259. Springer, 1985.

\bibitem{RuteSchnorr}
Kenshi Miyabe and Jason Rute.
\newblock Van {L}ambalgen’s theorem for uniformly relative {S}chnorr and
  computable randomness.
\newblock In {\em Proceedings of the 12th Asian Logic Conference}, pages
  251--270. World Scientific, 2013.

\bibitem{logicBlog2014}
Andre Nies.
\newblock Logic blog 2014.
\newblock {\em Preprint arXiv:1504.08163}, 2015.

\bibitem{ruteTalkCCR2015}
Jason Rute.
\newblock New directions in randomness, June 2015.
\newblock Slides http://math.uni-heidelberg.de/logic/conferences/ccr2015/.

\bibitem{vanLambalgenFails}
Liang Yu.
\newblock When van {L}ambalgen's theorem fails.
\newblock {\em Proceedings of the American Mathematical Society}, pages
  861--864, 2007.

\end{thebibliography}

\end{document}